\documentclass[12pt,twoside,sort&compress ]{article}
\usepackage{bbm}
\usepackage[numbers,sort&compress]{natbib}
\usepackage{mathrsfs}\usepackage{fontenc}\usepackage{textcomp}
\usepackage{graphicx}
\usepackage{amssymb}
\usepackage{graphicx}
\usepackage{extarrows}
\usepackage{chemarrow}
\usepackage{mathrsfs}
\usepackage{amsmath,amssymb,amsthm,amsxtra}
\usepackage{graphicx}
\usepackage{amssymb}
\usepackage[numbers,sort&compress]{natbib}
\usepackage{mathrsfs}\usepackage{fontenc}\usepackage{textcomp}
\usepackage{graphicx}
\usepackage{amssymb}
\usepackage{graphicx}
\usepackage{extarrows}
\usepackage{chemarrow}
\usepackage{mathrsfs}
\usepackage{amsmath,amssymb,amsthm,amsxtra}
\usepackage{graphicx}
\usepackage{amssymb}
 \numberwithin{equation}{section}
\newtheorem{theorem}{Theorem}[section]
\newtheorem{lemma}{Lemma}[section]
\newtheorem{definition}{Definition}[section]
\newtheorem{example}{Example}[section]

\newtheorem{corollary}{Corollary}[section]

\newtheorem{remark}{Remark}[section]
\begin{document}
\topmargin -5mm \oddsidemargin -1mm
\title{\Large \bf  On   classification of  singular matrix difference   equations of mixed order
\thanks{\footnotesize \textit{2010 Mathematicsg
Subject Classification}.   34B20, 39A27.
\newline ** The corresponding author. }
}
%revised June 19, 2004; Sept.15, Sept. 20, Nov. 27, final revision, 2004
\author{\small   Li Zhu$^{a}$, Huaqing Sun$^{b,\; ** \;}$,   Bing Xie$^{c}$
\\ \small $^{a,b,c}$  Department of Mathematics, Shandong University at Weihai
\\\small Weihai, Shandong 264209, P. R. China
\\\small $^a$ zhulimathematics@163.com
\\\small $^b$ sunhuaqing@email.sdu.edu.cn
\\\small $^c$ xiebing@sdu.edu.cn}

\date{}\baselineskip 20pt
\maketitle \ \ {\bf{Abstract}---\rm{
This paper is concerned with  singular matrix difference equations of mixed order.
The existence and uniqueness of  initial value problems for  these equations are derived, and then
the    classification  of  them  is obtained with a similar classical Weyl's method  by selecting a suitable quasi-difference.
An equivalent characterization of this classification is given in terms of the number of linearly independent square summable solutions of the equation.
The influence of off-diagonal coefficients on the classification is  illustrated by two examples.
In particular, two limit point  criteria are established in terms of coefficients of the equation.

}}\textsc{}\vspace{-5mm}
\\

\section{{{\normalsize\bf }{\sc\bf\normalsize
Introduction}}}
\ \ \ \
Consider the   matrix difference  expressions of mixed order:
$$
{\mathcal{L}}\left(\begin{array}{l}
y_1  \\
y_2 \end{array}\right)(t):=\left\{\begin{array}{l}
  \left(\begin{array}{cc}
-  \nabla p  \Delta+q  & -\nabla c +h \vspace{2mm} \\
c  \Delta+h  & d
\end{array}\right)\left(\begin{array}{l}
y_1  \\
y_2
\end{array}\right)(t),~ t\in \mathcal{I},\vspace{2mm}\\
\left(\begin{array}{cc}
0  & 0 \vspace{2mm} \\
c  \Delta+h  & d
\end{array}\right)\left(\begin{array}{l}
y_1  \\
y_2
\end{array}\right)(t),~ t=a-1,  \end{array}\right.
$$
where $\mathcal{I}:=\{t\}_{t=a}^{+\infty}$ is an integer set with $a$ being a  finite integer;
$\nabla$ and $\Delta$  are the backward and forward difference operators, respectively, i.e., $\nabla y(t)=y(t)-y(t-1) $ and $\Delta y(t)=y(t+1)-y(t)$;
$p$, $q,$ $c$, $ h,$ and $d$ are real-valued functions on $\mathcal{I}':=\{t\}_{t=a-1}^{+\infty}$ with  $p \neq 0$. Let
$\lambda$ be a spectral parameter. Then, equation $\mathcal{L}(y)=\lambda y$ on $\mathcal{I}'$
 can be expressed as follows:
\begin{equation}\label{00}
\left\{\begin{array}{l}
-  \nabla\left(p(t) \Delta y_{1}(t)\right) +q(t)  y_{1}(t)- \nabla \left(c(t) y_{2}(t)\right)+h(t) y_{2}(t)=\lambda y_1(t), \ t\in \mathcal{I}, \vspace{3mm} \\
c(t) \Delta y_{1}(t) +h(t) y_{1}(t)+d(t) y_{2}(t)=\lambda y_2(t), \ t\in {\mathcal{I^{\prime}}} .
\end{array}\right. \tag{${{1.1}_{\lambda}}$}
\end{equation}
In the case of $h=c\equiv 0$ on ${\mathcal{I^{\prime}}} $ and  $d(t)\neq0$ for $t\in {\mathcal{I^{\prime}}},$  the equations \eqref{00}   becomes the classical
Sturm-Liouville difference equations\vspace{-2mm}
\begin{equation}\label{3}\tau (y_1)(t):=-  \nabla \left(p(t) \Delta y_1(t)\right)  +  {q}(t)  y_1(t) =\lambda y_1(t),~ t\in \mathcal{I}. \tag{1.2}\vspace{-1mm}\end{equation}
Therefore, equations \eqref{00} contain classical Sturm-Liouville difference equations as their special ones. Moreover,
if  $y=\left(y_1, y_2\right)^{\mathrm{T}}$ (the superscript $\mathrm{T}$ denotes the transpose of a vector) satisfies \eqref{00},   then the first component
$y_1$ is a solution of the following  Sturm-Liouville difference equation with coefficients depending rationally on the spectral parameter:\vspace{-2mm}
\begin{equation}\label{002}
-  \nabla \left(\widetilde{p}(t, \lambda) \Delta y_1(t) \right) +\tilde{q}(t, \lambda) y_1(t) =\lambda y_1(t),~ t\in \mathcal{I}, \tag{${1.3}$}\vspace{-2mm}
\end{equation}
where $\widetilde{p}(t, \lambda)$   and  $\widetilde{q}(t,\lambda)$ are given by
\begin{equation}\label{003}
\begin{array}{lll}
&\widetilde{p}(t, \lambda):=p(t)+\dfrac{c^{2}(t)}{\lambda-d(t)}-\dfrac{h(t)c(t)}{\lambda-d(t)},~ t\in {\mathcal{I^{\prime}}}\vspace{2mm} \\
&\widetilde{q}(t,\lambda):=q(t)+\dfrac{h^{2}(t)}{\lambda-d(t)}-\nabla \left(
 \dfrac{h(t)c(t)}{\lambda-d(t)} \right),~ t\in \mathcal{I}. \tag{1.4}
\end{array}
\end{equation}
In addition,   $y_2$ can be expressed in terms of   $y_1$ as follows:
\begin{equation}\label{004}
y_2(t)=\frac{c(t)}{\lambda-d(t)} \Delta y_1(t)+\frac{h(t)}{\lambda-d(t)} y_1(t),~ t\in {\mathcal{I^{\prime}}} .\tag{1.5} \vspace{-1mm}
\end{equation}
Conversely, if $y_1$ and  $y_2$  satisfy  \eqref{002} and  \eqref{004},
then  $ y_1$ and $y_2$ is   a solution of \eqref{00}.
Hence, equation  \eqref{00} is  equivalent   to  \eqref{002} and   \eqref{004}   when  $\lambda$  is  given such that  \eqref{003} and \eqref{004} are well-defined.

Matrix differential expressions of mixed order arise in fluid mechanics, magnetohydrodynamics, and quantum mechanics, etc.
Essential spectra of operators generated by a class of $3\times3$ matrix differential expressions of mixed order for ideal magnetohydrodynamics models  were studied by  Kako in \cite{Kako}.
This work was generalized and developed by many authors (cf., e.g., \cite{Faierman1, Faierman2,Hardt,Mennicken,Mo}), and then spectral properties of this class of differential expressions were gotten more clear understanding. Up to now, the spectral theory for this class of  differential expressions   has been studied intensively (cf., \cite{Brown,Ibrogimov,Ibrogimov2,Kurasov1,Kurasov2,Kusche,Qi1,Qi2} and the references cited therein).
It is noted that most existing relevant results  are concerned with the following
$2\times2$ matrix differential equations of mixed order:
\begin{equation}\label{2222}
\mathfrak{L} \left(\begin{array}{l}
y_1  \\
y_2 \end{array}\right)(t):=
  \left(\begin{array}{cc}
-  D p D+q  & -Dc +h  \\
\overline{c} D+\overline{h}  & d
\end{array}\right)\left(\begin{array}{l}
y_1  \\
y_2
\end{array}\right)(t)=\lambda\left(\begin{array}{l}
y_1  \\
y_2 \end{array}\right)(t) ,~ t\in (a,b),\tag{1.6}
\end{equation}
where $-\infty<a< b\leq \infty$; $p^{-1}$, $q,$ $c$, $ h,$ and  $d$ are local integrable functions on $(a,b)$ with  $p \neq 0$; $D = {\rm d/dt},$
$\lambda$ is a spectral parameter.
Although equations \eqref{2222} are more simple forms of matrix differential expressions of mixed order,
they may contain more complicated examples including $3\times3$ ones which were considered, e.g.,
 in \cite{Faierman1, Faierman2,Hardt,Mennicken,Mo}, when $c(t), h(t)\in \mathbb{C}^n$, and  $d(t)\in  \mathbb{C}^{n\times n}$, $t\in (a,b)$, and
 $\overline{c}(t)$ and $\overline{h}(t)$ are replaced by $c^*(t)$ and $h^*(t)$, where $c^*(t)$ denotes the complex conjugate transpose of $c(t)$, $t\in (a,b)$.  Essential spectra of equations \eqref{2222} with the above vector
  and matrix coefficients have been investigated  by Ibrogimov, Siegl, and  Tretter  in great detail under considerably weaker assumptions \cite{Ibrogimov}.
For the study of non-self-adjoint matrix differential expressions, the reader can be referred to \cite{Ibrogimov2}.
At the same time, the spectral theory for abstract block operator matrices has been  developed and some elegant results have been established for the various  essential spectra, spectral decomposition, spectral enclosure, spectral inclusion, quadratic numerical range, and Friedrichs extension
 (cf., \cite{Tretter,Tretter1,Langer,Langer1,Konstantinov,Jeribi,Jeribi2, Giribet,Atkinson,Ammar3}).
As everyone knows,  there are a large number of discrete mathematical models in applications. The spectral theory of
discrete  systems has attracted a great deal of interest (cf., \cite{Atkinson0,Jirari,Mingarelli,shiyu,Clark,P,Monaquel,Hinton} and the references cited therein).
 Equations \eqref{00} can be regarded as a discrete  analogue  of  the singular  equations \eqref{2222}.
However, as far as we know, there are little attention on
 equations \eqref{00} including the regular case and the singular case.

For classical differential operators, the Weyl-Titchmarsh theory is extremely useful in the spectral analysis,
which goes back to  H. Weyl's work  \cite{Weyl}. He initially classified
singular second-order symmetric  differential equations  into two cases: the limit point case and the limit circle case, based on geometrical properties of a certain limiting set.
In the limit circle case, the essential spectrum of the associated operator is empty. In addition, this classification is closely related to characterizations of
 self-adjoint extensions of the minimal operators generated by symmetric differential expressions. This work was followed and developed extensively and intensively and many good results have been obtained
for differential and difference expressions including symmetric and non-symmetric cases (cf., e.g., \cite{Behrndt11,   Behrndt1,Brown1, Coddington,  Hinton2, Titchmarsh,Jirari,shiyu,sun,Krall,Krall1,Monaquel,Muzzulini,Weidmann,P,Kogan,Behrndt}).
Some limit point and limit circle criteria have been established for singular differential and difference expressions \cite{Chen,Coddington,Jirari,sun,Weidmann,Mingarelli,Levinson,Everitt,Evans}. Sturm-Liouville  differential equations with coefficients depending rationally on the spectral parameter attracted people's interest in the past because of their floating singularities (cf., e.g., \cite{Hassi0,Atkinson2,Adamyan}).
Also, there is an analog of the limit point  and   limit circle classification for this class of singular differential equations \cite{Hassi0}.
Furthermore, equation \eqref{2222} satisfying certain conditions can also be classified into the limit point case and the limit circle case by transforming them into symmetric Hamiltonian systems \cite{Qi1,Qi2}. Especially,  a similar  classification has been made for more general equations \eqref{2222} with real coefficients  by using  classical Weyl's method \cite{Hassi}.

Unlike classical differential operators, singular matrix differential operators have their interesting and unexpected spectral properties.
Their essential spectrum
consists of two parts: a regular part and  a singular part (cf. e.g., \cite{Kurasov1, Ibrogimov, Hardt, Faierman2}). Since the second part appears  due to singularities of coefficients at the endpoint, it is empty when
the matrix differential expression is in the limit circle case. Therefore, this classification  is crucial for the study of spectral properties of singular matrix differential  equations.
Inspired by the work of \cite{Hassi}, we shall consider the  classification  of equations \eqref{00} by using  Weyl's method, and investigate spectral properties of them  in the subsequent study.
In this paper,  the existence and uniqueness of  initial value problem of  equation \eqref{00}  are derived, and then
the   classification  is obtained by selecting a suitable quasi-difference.
Similarly to classical differential systems, it is proved that an equivalent characterization of this classification can also be given
in terms of the number of linearly independent square summable solutions of equation \eqref{00}.
The influence of off-diagonal coefficients $c(t)$ and $h(t)$ on this classification is illustrated by two examples.
In particular, two  limit point criteria are established in terms of coefficients of equation \eqref{00}.

The paper is organized as follows.
In Section 2,  the Green's formula  and the existence and uniqueness of  initial value problem for equation \eqref{00} are derived.
In Section 3,  the  classification is shown, and the equivalent characterization  is given.
Section 4 is devoted to the influence of off-diagonal  coefficients on  this classification.
Section 5 gives  two limit-point criteria.

\section{{{\sc\normalsize\bf Preliminaries
}}}\baselineskip 20pt
\ \ \ \ In this section,   the Green's formula for $\mathcal{{L}}$ or $({1.1}_{\lambda})$ is obtained, and   the existence and uniqueness of  initial value problem for  $({1.1}_{\lambda})$ are derived.

First, let $l (\mathcal{I})=\big\{y =   {\{y(t)\}}^{+\infty}_{t=a-1} \subset \mathbb{C}^{2} \big\}  $ and  then
we introduce the following  space:
$$ l^{2}(\mathcal{I}): = \left\{y \in l (\mathcal{I}): \sum_{t \in \mathcal{I}}  y^{*}(t)  y(t)<+\infty \right\}$$
with the  inner product
${\langle y, z\rangle}: = \sum_{t \in \mathcal{I}} z^{*}(t)  y(t),$ where $ z^{*}$ denotes the complex conjugate transpose  of $ z $.
The induced norm is  ${\|y\|} : ={ {\langle y, y\rangle}}^{\frac{1}{2}} $  for $  y \in  l^{2}(\mathcal{I}).$
For  $ N \in \mathcal{I}\setminus \{a\},$  set
${\mathcal{I}}_N:=\{t\}_{t=a}^{N}$,
$$l (\mathcal{I}_{N})=\big\{y =   {\{y(t)\}}^{N+1}_{t=a-1} \subset \mathbb{C}^{2} \big\},$$
and let the definition of  $ l^{2}({\mathcal{I}}_N)$ be    similar to that of $ l^{2}({\mathcal{I}})$ with ${\mathcal{I}} $ replaced by $ {\mathcal{I}}_N.$ By${\langle\cdot, \cdot\rangle}_{N}  $ and  ${\|\cdot\|}_{N}, $  we denote the inner product and its induced norm of $ l^{2}({\mathcal{I}}_N)$.

Next, for  $y =(y_{1} , y_{2}  )^{\mathrm{T}}\in l (\mathcal{I})$,  the quasi-difference operator $y^{[1]}(t)$  is defined by  \vspace{-2mm}
\begin{equation}\label{0041}
y^{[1]}(t):=p(t)\Delta y_1(t)+c(t)y_2(t), ~ t \in {\mathcal{I^{\prime}}}.
\end{equation}
Further,       the Lagrange bracket of $y =(y_{1} , y_{2}  )^{\mathrm{T}}$ and $z =(z_{1} , z_{2}  )^{\mathrm{T}} $  is defined by
\begin{equation}\label{005}
[y(t), z(t)]:=y_{1}(t+1) \overline{z}^{[1]}(t)-y^{[1]}(t) \overline{z}_{1}(t+1), ~ t \in {\mathcal{I^{\prime}}}.\vspace{-2mm}
\end{equation}
Then, the Green's formula for $\mathcal{{L}}$ or $({1.1}_{\lambda})$ can be given as follows.
\begin{lemma}\label{21}
For   $y, z \in l ({\mathcal{I}}_N)$, it holds that
\begin{equation}\label{006}
\langle \mathcal{{L}}(y), z\rangle_{N}- \langle y, \mathcal{{L}}(z)\rangle_{N} ={[y(t), z(t)]\Big|_{t=a-1}^{N}}. \vspace{-2mm}
\end{equation}
\end{lemma}

\begin{proof}
Using \eqref{0041}, we have
\begin{equation*}
\begin{aligned}
&\langle \mathcal{{L}}(y), z\rangle_{N}- \langle y, \mathcal{{L}}(z)\rangle_{N}\\
=&
\sum_{t \in \mathcal{I}_N}  \Big\{\left(\overline{z}_{1}, \overline{z}_{2} \right)\Big(\begin{array}{c}
-  \nabla \left(p \Delta y_{1}+c y_{2}\right) +q  y_{1}+h y_{2} \\ \vspace{-2mm}
c \Delta y_{1} +h y_{1}+d y_{2}
\end{array}\Big)(t)\\
&   -  \Big( -  \nabla \left(p \Delta \overline{z}_{1}+c \overline{z}_{2}\right) +q \overline{ z}_{1}+h \overline{z}_{2}, c \Delta \overline{z}_{1} +h \overline{z}_{1}+d \overline{z}_{2}
\Big)  \Big( \begin{array}{c} {y}_{1}  \\    {y}_{2}  \end{array}\Big) (t) \Big\}
\\
=& {  \Big[p(t) \big(y_1(t) \Delta  \overline{{z}}_{1}(t) -  \overline{{z}}_{1}(t) \Delta  y_1(t) \big)}
 {  + c(t)\big(y_1(t+1)  \overline{{z}}_{2}(t)- y_2(t)   \overline{{z}}_{1}(t+1)\big)  \Big]}\Big|_{t=a-1}^{N}\\
= & \Big[y_{1}(t+1) \overline{z}^{[1]} (t)- y^{[1]}(t)\overline{z}_{1}(t+1)\Big] \Big|_{t=a-1}^{N}\\ \vspace{-2mm}
= &[y(t), z(t)]\Big|_{t=a-1}^{N}.
\end{aligned}
\end{equation*}
This completes the proof.
\end{proof}

The following  lemma  is a consequence of the Green's formula given by \eqref{006}.
\begin{lemma}\label{consequence}
Let  $\varphi(t, \lambda)$ be a solution of $({1.1}_{\lambda})$ and   $\psi(t, \mu)  $  a solution of $({1.1}_{\mu})$. Then\vspace{-2mm}
\begin{equation}\label{020}
{(\lambda-\overline{\mu})} \sum_{t \in \mathcal{I}_N}    {\psi^*(t, \mu)}\varphi(t, \lambda)   = {[\varphi(t, \lambda), \psi(t, \mu)]\big|_{t=a-1}^{N}}
\end{equation}
holds for all $N \in \mathcal{I}$.
\end{lemma}

\begin{proof}
Since $\mathcal{{L}}(\varphi)=\lambda \varphi $ and $\mathcal{{L}}(\psi)=\mu \psi,$
we have \vspace{-2mm}
$$ \langle\mathcal{{L}}(\varphi), \psi\rangle_{N}- \langle \varphi, \mathcal{{L}}(\psi)\rangle_{N}= (\lambda-\overline{\mu}) \sum_{t \in \mathcal{I}_N}    {\psi^*(t, \mu)}\varphi(t, \lambda).$$
Then, \eqref{020} holds by Lemma \ref{21}.
This completes the proof.
\end{proof}

Now, set  $\sigma(f) :=\{\lambda\in \mathbb{R}:  \inf_{t\in {\mathcal{I^{\prime}}} }|\lambda- f(t)|=0\} $ and    $\Omega(f):=  \mathbb{C} \setminus \sigma(f)$  for a  function  $f(t),\ t\in {\mathcal{I^{\prime}}}.$
If $\lambda\in \Omega(d)$, then $\widetilde{p}(t, \lambda)$ is well-defined on $\mathcal{I^{\prime}}.$
Further, if $y =(y_{1} , y_{2}  )^{\mathrm{T}}$ is a solution of \eqref{00} with  $\lambda\in \Omega(d)$, then by \eqref{004} and \eqref{0041},
we get
\begin{equation}\label{291}
 {y}^{[1]}(t)=  \widetilde{p}(t, \lambda)\Delta{y}_{1}(t)+ \alpha(t){y}_{1}(t+1), ~ t \in \mathcal{I^{\prime}},\vspace{-2mm}
\end{equation}
where $\alpha(t)$ is given by \vspace{-2mm}  $$\alpha(t)=\dfrac{  h(t) c(t) }{\lambda-d(t)},~t \in \mathcal{I^{\prime}}. $$

\begin{lemma}\label{constant}
   Let $ \varphi =\left(\varphi_1, \varphi_2\right)^{\mathrm{T}}$ and $\psi=\left(\psi_1, \psi_2\right)^{\mathrm{T}}$  be solutions of \eqref{00}. Then, for $\lambda\in \Omega(d)$ and $t\in \mathcal{I^{\prime}},$
\begin{equation}\label{022}
 [\psi(t,\lambda), \overline{\varphi}(t,\lambda)]=\widetilde{p}(t, \lambda) \Big( {\psi}_{1}(t+1,\lambda)\Delta{\varphi}_{1}(t,\lambda)- {\varphi}_{1}(t+1,\lambda) \Delta{\psi}_{1}(t,\lambda)\Big),
\end{equation}
and further,  $[\psi(t,\lambda), \overline{\varphi} (t,\lambda)]$ is a   constant on $\mathcal{I^{\prime}}$.
\end{lemma}

\begin{proof}
Let $\lambda\in \Omega(d)$ and $ \varphi   =\left(\varphi_1 , \varphi_2 \right)^{\mathrm{T}}$ and $\psi =\left(\psi_1 , \psi_2 \right)^{\mathrm{T}}$  be solutions of \eqref{00}.
Then  \eqref{002}  and \eqref{291} hold for  $ \varphi$ and $\psi$, respectively.
Then,  \eqref{022}  can be  easily obtained by \eqref{291}.  Moreover, using \eqref{002} and \eqref{022}, one has
$$
\begin{aligned}
\nabla  [\psi(t,\lambda), \overline{\varphi} (t,\lambda) ] &=\psi_{1}(t,\lambda)\nabla \big( \widetilde{p} (t, \lambda)\Delta \varphi_{1}(t,\lambda)  \big)-\varphi_{1}(t,\lambda)\nabla \big( \widetilde{p}(t, \lambda)\Delta \psi_{1}(t,\lambda)  \big)\\
&=\psi_{1}(t,\lambda)(\widetilde{q}(t, \lambda)-\lambda)\varphi_{1}(t,\lambda)-\varphi_{1}(t,\lambda)(\widetilde{q}(t, \lambda)-\lambda)\psi_{1}(t,\lambda)=0,~t\in \mathcal{I},\vspace{-2mm}
\end{aligned}
$$
which implies that   $[\psi(t,\lambda), \overline{\varphi }(t,\lambda)]$ is a constant on $\mathcal{I^{\prime}}$. This completes the proof.
\end{proof}

For convenience, let\vspace{-2mm}
$$\mathcal{M}(t):=d(t)-\dfrac{ c^{2}(t)-h(t)c(t)}{p(t)}, \quad~t \in {\mathcal{I^{\prime}}},\quad {\Omega}^{\prime}(\mathcal{M},d):=\mathbb{C} \backslash ( \sigma(\mathcal{M})\cup \sigma(d)).$$
Note that  $\widetilde{p} (t, \lambda)=\dfrac{p(t)}{\lambda-d(t)}\big( \lambda- \mathcal{M}(t)\big).$ Then, $\widetilde{p}(t, \lambda)$ and  $\widetilde{p}^{-1}(t, \lambda)$ are well-defined for every $t\in {\mathcal{I^{\prime}}} $ in the case of  $\lambda \in {\Omega}^{\prime}(\mathcal{M},d).$
The existence and uniqueness of initial value problems  for \eqref{00} are given below.

\begin{theorem}\label{initial}
Let $\lambda \in {\Omega}^{\prime}(\mathcal{M},d)$ and   $c_{1}, c_{2} \in \mathbb{C}$. Then the initial value problem \vspace{-2mm}
\begin{equation}\label{2069}
 \mathcal{{L}}(y)(t)=\lambda  y(t), \ t\in \mathcal{I}^{\prime}, \ \quad y_{1}(a)=c_{1}, \quad y^{[1]}(a-1)=c_{2},
\end{equation}
has a unique solution $y=y(t, \lambda)$ on $ {\mathcal{I^{\prime}}} $.
\end{theorem}

\begin{proof}
Let $\lambda \in {\Omega}^{\prime}(\mathcal{M},d)$ and  $y=\left(y_1, y_2\right)^{\mathrm{T}}$  satisfy \eqref{2069}.
Then,   \eqref{002} and \eqref{291} hold. It follows from \eqref{291} that
\begin{equation}\label{29}
\Delta y_{1}(t) =-\frac{\alpha(t)}{ \widetilde{p}(t, \lambda)} y_{1}(t+1)+\frac{1}{\widetilde{p}(t, \lambda)} y^{[1]}(t), ~ t \in {\mathcal{I^{\prime}}}.
\end{equation}
Using \eqref{002},  \eqref{291},   \eqref{29}, and $-\Delta y_1(t)+y_{1}(t+1)=y_{1}(t ),$  one has
\begin{equation}\label{210}
\begin{aligned}
\nabla y^{[1]}(t) &=\nabla \left(\widetilde{p}(t, \lambda) \Delta y_{1}(t) \right) +  y_{1}(t) \nabla \alpha(t)+\alpha(t) \Delta{y}_{1}(t) \vspace{2mm}\\
&=H(t) y_{1}(t)+\alpha(t)\left(  -\frac{\alpha(t) }{ \widetilde{p}(t, \lambda)} y_{1}(t+1)+\frac{1}{\widetilde{p}(t, \lambda) } y^{[1]}(t)    \right)\vspace{2mm}\\
&=\left( \frac{\alpha(t)H(t)- \alpha^2(t) }{\widetilde{p}(t, \lambda) }  +H(t) \right) y_{1}(t+1) + {\frac{ \alpha(t) - H(t)  }{ \widetilde{p}(t, \lambda) }}     y^{[1]}(t),~ t \in \mathcal{I},
\end{aligned}
\end{equation}
where $$H(t):= \widetilde{q}(t)+\nabla \alpha(t) -\lambda=q(t)+\dfrac{   h^{2}(t) }{\lambda-d(t)}-\lambda, ~ t \in \mathcal{I}.$$

Now, let
$$A(t, \lambda)=\left(\begin{array}{ll}- \alpha(t)  \widetilde{p}^{-1}(t, \lambda)  & \widetilde{p}^{-1}(t, \lambda)\vspace{4mm} \\  (H(t)-{\alpha(t)} )\alpha(t)   \widetilde{p}^{-1}(t, \lambda) +H(t) & ( \alpha(t)- H(t) )\widetilde{p}^{-1}(t, \lambda)\end{array}\right).$$
By a simple calculation,   $\rm{det}$$(I_2-A(t,\lambda))\equiv 1 $ on  $  \mathcal{I}$, where $I_2\in {\mathbb{C}}^{2\times2}$  is the  identity matrix.
Note that $\Delta{y}_{1}(t)= \nabla{y}_{1}(t+1).$
Then,   by \eqref{29} and \eqref{210},  one has the first-order system of difference equations
\begin{equation}\label{212}
\nabla \left(\begin{array}{c}
 y_{1}(t+1) \\
  y^{[1]}(t)
\end{array}\right) =A(t, \lambda)\left(\begin{array}{c}
y_{1}(t+1) \\
y^{[1]}(t)
\end{array}\right), ~t\in  \mathcal{I}.
\end{equation}
On the other hand,   using    \eqref{29} and $-\Delta y_1(t)+y_{1}(t+1)=y_{1}(t )$ again, we get from    \eqref{004} that for $t \in {\mathcal{I^{\prime}}}$,
\begin{equation}\label{2144}
 \begin{aligned}
y_2(t)= \left[\frac{\alpha(t) (h(t)- c(t))}{(\lambda-d(t))\widetilde{p}(t, \lambda)}
+ \frac{ h(t)}{\lambda-d(t)} \right]y_{1}(t+1)
 + \frac{c(t)-h(t)}{(\lambda-d(t))\widetilde{p}(t, \lambda)}y^{[1]}(t).
\end{aligned}
\end{equation}
It can  be  easily verified that  $y=\left(y_1, y_2\right)^{\mathrm{T}}$ satisfies  \eqref{00}  if and only if \eqref{212} holds on $\mathcal{ I}$ and  \eqref{2144} holds  on  ${\mathcal{I^{\prime}}}  $.

Since   $\rm{det}$$(I_2-A(t,\lambda))\equiv 1 $ on  $  \mathcal{I}$,
the fundamental matrix $Y(t, \lambda)$ of \eqref{212}   with $Y(a, \lambda)=I_{2}$   exists    uniquely on  $  \mathcal{I}$. Next, for  $c_{1}, c_{2} \in \mathbb{C}$,    let $y_1(a)=c_1$ and $y^{[1]}(a-1)=c_2$,
  $y_1(t)$ for $t \in {\mathcal{I}}\setminus\{a\}$ be determined by $\left(y_{1}(t+1), y^{[1]}(t)\right)^{\mathrm{T}}=Y(t, \lambda)\left(c_{1}, c_{2}\right)^{\mathrm{T}},$ $t \in \mathcal{I}$,
and    $y_2 $  satisfy \eqref{2144} on $\mathcal{I}$. Then  $y=\left(y_1, y_2\right)^{\mathrm{T}}$ is given on $\mathcal{I}$.
In addition, it follows from    \eqref{004} and \eqref{0041} that
\begin{equation} \label{8310}
\left\{\begin{array}{l}
\dfrac{h(a-1)-c(a-1)}{\lambda-d(a-1)}y_1(a-1)-y_2(a-1) =\dfrac{-c(a-1)}{\lambda-d(a-1)}c_1, \vspace{2mm}   \\
-p(a-1)y_1(a-1)+c(a-1) y_2(a-1)  =c_2-p(a-1)c_1.
\end{array}\right.\vspace{-1mm}
\end{equation}
Since the coefficient determinant of \eqref{8310} is   not zero by $\lambda \in {\Omega}^{\prime}(\mathcal{M},d),$ \eqref{8310} has a unique  solution $(y_1(a-1),y_2(a-1))^{\mathrm{T}}$,
with which  $y=\left(y_1, y_2\right)^{\mathrm{T}}$ is  the unique solution satisfying \eqref{2069}.
This completes the proof.
\end{proof}

For $\lambda \in {\Omega}^{\prime}(\mathcal{M},d),$  it follows from Theorem \ref{initial} that
mapping $y \mapsto\left(y_{1}(a), y^{[1]}(a-1)\right)^{\mathrm{T}} \in \mathbb{C}^{2}$ is bijective for solutions  $y$ of equation \eqref{00}. Hence, we have
\begin{corollary}\label{dimension}
For $\lambda \in {\Omega}^{\prime}(\mathcal{M},d)$,   the set of solutions   of equation \eqref{00} is a vector space of dimension $2$.
\end{corollary}

\section{{\large{\bf  Classification   of  singular matrix difference   equations of mixed order}}}{}\baselineskip 20pt
\ \ \  \
It has been known that singular Sturm-Liouville differential and difference  equations can be classified  into the
limit point case  and the limit circle case, respectively,  by the Weyl's method, i.e.,  in terms of geometrical properties of the limiting set of a sequence of nested Weyl's circles \cite{Weyl}.
This work has been developed intensively, and especially, Hassi et al \cite{Hassi} founded that matrix differential equations \eqref{2222} can also be classified with a similar method
when the coefficients are real-valued.
Motivated by the work given by \cite{Hassi}, we shall give the   classification for equation $(1.1_\lambda)$ or $\mathcal{L}$ by constructing  nested Weyl's circles.
This section consists of two subsections.

\subsection{\normalsize\bf  Weyl's circles and limit point and limit circle  classification}
\ \ \ \
In this subsection, we shall construct Weyl's circles for equations $(1.1_\lambda)$ or $\mathcal{L}$ over finite intervals,
which are nested and converge to a limiting set. Equations $(1.1_\lambda)$ or $\mathcal{L}$ will be classified in terms of properties of the limiting set.
We present it in detail for the convenience of the reader.

First, we consider  $\mathcal{L}$ on the finite interval
${\mathcal{I^{\prime}}}_N:=\{t\}_{t=a-1}^{N+1}, ~N \in \mathcal{I}\setminus \{a\}$,
satisfying the following  boundary conditions:\vspace{-1mm}
\begin{equation}\label{31001}
\left\{\begin{array}{l}
U_1(y):=y_1(a) \sin \alpha- y^{[1]}(a-1) \cos \alpha=0 , \vspace{2mm}\\
U_2(y):=y_1(N+1) \cos \beta+y^{[1]}(N )  \sin \beta=0,  \\
   \end{array}\right.\vspace{-1mm}
\end{equation}
where $ 0 \leqslant \alpha, \beta<\pi.$
The boundary value problem  \eqref{00} with \eqref{31001} is
called a regular one since it defined on a finite interval.
By Theorem 2.1, for $\lambda \in {\Omega}^{\prime}(\mathcal{M},d)$, let  $\varphi(t, \lambda)=\left(\varphi_{1}(t, \lambda), \varphi_{2}(t, \lambda)\right)^{\mathrm{T}}$ and $\psi(t, \lambda)=\left(\psi_{1}(t, \lambda), \psi_{2}(t, \lambda)\right)^{\mathrm{T}}$  be solutions of \eqref{00}
with  the initial conditions:
\begin{equation}\label{case}
\begin{aligned}
 \varphi_{1}(a, \lambda)=\sin \alpha, ~& \varphi^{[1]}(a-1, \lambda)=-\cos \alpha;\\
\psi_{1}(a, \lambda)=\cos \alpha,~ & \psi^{[1]}(a-1, \lambda)=\sin \alpha.
\end{aligned}
\end{equation}
From   Lemma \ref{constant} and    \eqref{case}, we can get that   $\widetilde{\varphi} (t,\lambda):=(\varphi_{1}(t+1,\lambda), \varphi^{[1]}(t,\lambda) ) ^{\mathrm{T}}$ and  $\widetilde{\psi} (t,\lambda):=(\psi_{1}(t+1,\lambda), \psi^{[1]}(t,\lambda))^{\mathrm{T}}$  are linearly independent on  $ {{\mathcal{I}}^{\prime}}_N .$
Then, we claim that   $\varphi(t,\lambda)$ and  $\psi(t,\lambda)$  are   linearly independent on ${{\mathcal{I}}^{\prime}}_N.$
In fact, suppose on the contrary that  $\varphi(t,\lambda)$ and  $\psi(t,\lambda)$ are  linearly  dependent, i.e.,     $\varphi(t,\lambda)=k\psi(t,\lambda), t\in {{\mathcal{I}}^{\prime}}_N $,  for some $k.$
Then
$ \varphi^{[1]}(t,\lambda)=k \psi^{[1]}(t,\lambda) $ by \eqref{0041}, which yields that
 $\widetilde{\varphi} (t,\lambda)$  and $\widetilde{\psi} (t,\lambda)$ are linearly  dependent on ${{\mathcal{I}}^{\prime}}_N,$  which is a
contradiction. Hence,  $\varphi(t,\lambda)$ and  $\psi(t,\lambda)$ are  linearly  independent on ${{\mathcal{I}}^{\prime}}_N.$
Furthermore, the following result holds:

\begin{lemma}\label{311311}
A number  $\lambda \in {\Omega}^{\prime}(\mathcal{M},d)$ is an eigenvalue of  boundary value problem   \eqref{00} with   \eqref{31001} if and
only if $ U_2(\psi)=0 .$
\end{lemma}

\begin{proof}  It is evident that $U_1(\psi)=0$ for $\lambda \in {\Omega}^{\prime}(\mathcal{M},d)$ by \eqref{case}.
Therefore, in the case that $U_2(\psi)=0$, this $\lambda$ is an eigenvalue of \eqref{00} with   \eqref{31001} and $\psi$ is the associated eigenvector.
Hence, the sufficiency is proved.

Now, we show the necessity.
Suppose that   $\lambda \in {\Omega}^{\prime}(\mathcal{M},d)$ is an eigenvalue of    \eqref{00} with   \eqref{31001} and   $y= \left(y_{1}, y_2 \right)^{\mathrm{T}}$ be the associated  eigenvector.
Then  $y$  can be expressed as $y(t,\lambda)=c_1\varphi(t,\lambda)+c_2\psi(t,\lambda)$, $t \in {{\mathcal{I}}^{\prime}} $, $c_1,c_2\in \mathbb{C}$, since it is a solution of \eqref{00}.
Inserting this expression of $y$ into $U_1(y)=0$ and using \eqref{case}, we get that $c_1=0$.
Furthermore, inserting this expression of $y$  with $c_1=0$ into $U_2(y)=0$ and using \eqref{case} again, we have
$c_2 U_2(\psi)=0.$
Since $y$ is a nontrivial solution of  \eqref{00}, we have $ c_2\neq0.$ Hence,  it follows from the above relation that  $ U_2(\psi)=0.$
This completes the proof.
\end{proof}

By   iterative calculations,
 $\psi_{1}(N+1,\lambda )$ and $\psi^{[1]}(N+1,\lambda )$ with $\lambda \in {\Omega}^{\prime}(\mathcal{M},d)$ are rational fraction polynomials  of  $\lambda$ with finite  terms, respectively.
Therefore, from Lemma  \ref{311311} and definition of $ U_2(\psi)$,
the boundary value problem   \eqref{00} with   \eqref{31001} has  eigenvalues  and the number of all the eigenvalues is finite.
In addition, for $y =(y_{1} , y_{2}  )^{\mathrm{T}}$ and $z =(z_{1} , z_{2}  )^{\mathrm{T}}  \in l ({\mathcal{I}}_N)$  satisfying \eqref{31001},
 it can be verified that \vspace{-2mm} $$ [y(a-1), z(a-1)]=[y(N), z(N)].\vspace{-2mm}$$
 Then,  it follows  that
$\langle \mathcal{{L}}(y), z\rangle_{N}= \langle y, \mathcal{{L}}(z)\rangle_{N} $ by Lemma \ref{21},
which implies that boundary value problem \eqref{00} with \eqref{31001} is symmetric in the space $l^2(\mathcal{I}_N)$.
Hence, all  eigenvalues of boundary value problem   \eqref{00} with   \eqref{31001} are   real numbers.

For simplicity,  for $\lambda \in {\Omega}^{\prime}(\mathcal{M},d),$ let
$$  A=\varphi_{1}(N+1, \lambda), ~ B=\varphi^{[1]}(N, \lambda), ~ C=\psi_{1}(N+1, \lambda), ~D=\psi^{[1]}(N, \lambda) .$$
Then, it is evident that
\begin{equation}\label{34}
A\bar{D}-B \bar{C} =[\varphi(N,\lambda),\psi(N,\lambda)],\ C\bar{D}-\bar{C}D=[\psi(N,\lambda),\psi(N,\lambda)].
\end{equation}
By \eqref{34} and Lemma 2.2, we have\vspace{-2mm}
\begin{equation}\label{1351352}C\bar{D}-\bar{C}D =2i{\rm Im \lambda}\sum_{t \in \mathcal{I}_N}    {\psi^*(t, \lambda)}\psi(t, \lambda),\ i=\sqrt{-1}.\vspace{-2mm}\end{equation}
 In addition,  by Lemma \ref{constant} and   \eqref{case}, we have\vspace{-2mm}
\begin{equation}\label{344}
 A D-B C=[\varphi(N, \lambda), \overline{ \psi} (N, \lambda)]=1.\vspace{-2mm}
\end{equation}

Now, let $\lambda\in \mathbb{C}\setminus \mathbb{R}$.
Then   $\lambda \in {\Omega}^{\prime}(\mathcal{M},d) $ since $\mathcal{M}$ and $d$ are real-valued.
Hence,   $\lambda$ is not an eigenvalue of \eqref{00} with   \eqref{31001} since ${\rm Im}\lambda\neq0$.
Therefore, $U_2(\psi)\neq 0$ by Lemma 3.1 since $U_1(\psi)=0$.
Next, let $\chi(t,\lambda )=\left(\chi_1(t,\lambda ), \chi_2(t,\lambda )\right)^{\mathrm{T} }$   be given by\vspace{-2mm}
\begin{equation}\label{03707}
\chi(t, \lambda ):=\varphi (t, \lambda)+m(\lambda) \psi (t, \lambda),\ t\in{{\mathcal{I}}^{\prime}}.  \vspace{-2mm}
\end{equation}
Then $\chi(t,\lambda )$ is a   solution   of  equation \eqref{00}.  Let  $\chi(t,\lambda )$ satisfy  the boundary condition  $U_2(\chi(\cdot,\lambda ))=0.$
 Then we  get\vspace{-2mm}
\begin{equation}\label{32}
m(\lambda)=-\dfrac{U_2(\varphi)}{U_2(\psi)} =-\dfrac{A z+B}{C z+D},\vspace{-2mm}
\end{equation}
where $z=\cot \beta$, $ 0 \leqslant \beta<\pi.$
It is noted that \eqref{32} describes a circle, denoted by $C_N$, in the complex plane as $z $ varies.   We shall give the characteristics of the  circle $C_N$ below.

\begin{theorem}\label{mainlemma}
The center $O_{N}$ and radius $r_{N}$ of circle $C_N$ are respectively given by
\begin{equation}\label{9312}
O_{N}=-\frac{[\varphi(N, \lambda), \psi(N, \lambda)]}{[\psi(N, \lambda), \psi(N, \lambda)]}, \quad r_{N}=\frac{1}{\big|[\psi(N, \lambda), \psi(N, \lambda)] \big|},
\end{equation}
and further,  the equation and   interior  of circle  $C_N$ are respectively  given by
\begin{equation}\label{310}
\sum_{t \in \mathcal{I}_N} \chi^*(t, \lambda ) \chi(t, \lambda ) =\frac{\operatorname{Im} m }{\operatorname{Im} \lambda},
\quad \sum_{t \in \mathcal{I}_N} \chi^*(t, \lambda ) \chi(t, \lambda )<\frac{\operatorname{Im} m }{\operatorname{Im} \lambda}.
\end{equation}

\end{theorem}

\begin{proof}
Let $\chi(t,\lambda ) $ be defined by \eqref{03707}
satisfying  $U_2(\chi(\cdot,\lambda ))=0$. Then, by the fact that $\cos\beta$ and $\sin\beta$ are real numbers, it can be concluded that
$\chi_{1}(N+1, \lambda ) \overline{\chi}^{[1]}(N, \lambda)$ is a real number.
Since it holds that \vspace{-2mm}$$[\chi(N, \lambda ), \chi(N, \lambda )]=2 i \operatorname{Im} \left\{\chi_{1}(N+1, \lambda ) \overline{\chi}^{[1]}(N, \lambda)\right\},\vspace{-2mm}$$
we have\vspace{-2mm}
\begin{equation}\label{38}
[\chi(N, \lambda ), \chi(N, \lambda )]=0.
\end{equation}
Conversely, if \eqref{38} holds, then $\chi_{1}(N+1, \lambda ) \overline{\chi}^{[1]}(N, \lambda) $ is a real number.
Hence, from \eqref{38} and the fact that $\chi_{1}(N+1, \lambda ) \overline{\chi}^{[1]}(N, \lambda)$ is a real number, it can be verified that there exists $\beta\in [0,\pi)$ such that $U_2(\chi(\cdot,\lambda ))=0$.

Note that $C\bar{D}-\bar{C}D \neq 0 $ by \eqref{1351352}. Then, by using \eqref{34} and the definition of $\chi(t, \lambda )$,  we have
\begin{equation}\label{1351351}
\frac{[\chi(N, \lambda ), \chi(N, \lambda )]}{{C \bar{D}-\bar{C} D}}=m \bar{m}-\frac{\bar{A} D-\bar{B} C}{C \bar{D}-\bar{C} D} m+\frac{A \bar{D}-B \bar{C}}{C \bar{D}-\bar{C} D}\bar{m}-\frac{ \bar{A}B  - \bar{B}A}{C \bar{D}-\bar{C} D}.
\end{equation}
As a result,  we get from \eqref{38} that
\begin{equation}\label{135135}
m \bar{m}-\frac{\bar{A} D-\bar{B} C}{C \bar{D}-\bar{C} D} m+\frac{A \bar{D}-B \bar{C}}{C \bar{D}-\bar{C} D}\bar{m}=\frac{ \bar{A}B  - \bar{B}A}{C \bar{D}-\bar{C} D}.
\end{equation}
Equation \eqref{135135} gives a clear expression of circle $C_{N}$. From  \eqref{135135} and \eqref{34}, the center $O_{N}$   of  circle $C_N$ is given by
$$
O_{N}=\frac{B \bar{C}- A \bar{D}}{C\bar{D}  -  \bar{C} D}=-\frac{[\varphi(N, \lambda), \psi(N, \lambda)]}{[\psi(N, \lambda), \psi(N, \lambda)]}.
$$
Further, it follows from  \eqref{344} and   \eqref{135135}  that
\begin{align*}
{r_{N}}^{2}&={|O_{N}|}^{2}+\frac{ \bar{A}B  - \bar{B}A}{C \bar{D}-\bar{C} D}
 ={\left|\frac{ A D-B C }{C \bar{D}- \bar{C} D }\right|}^{2}= \frac{1}{{|[\psi(N, \lambda), \psi(N, \lambda)]|}^{2}}.
\end{align*}
This completes the proof of \eqref{9312}.

In addition, it is easy to verify that\vspace{-2mm}
\begin{equation*}
[\chi(a-1, \lambda ), \chi(a-1, \lambda)] =- 2i \operatorname{Im} m.  \vspace{-2mm}
\end{equation*}
Thus,  it follows from Lemma \ref{consequence}  that
$$\vspace{-2mm}
\begin{aligned}
{\left[\chi(N, \lambda ), \chi(N, \lambda )\right] } &=2 i \operatorname{Im} \lambda \sum_{t \in \mathcal{I}_N}  \chi^*(t, \lambda ) \chi(t, \lambda ) +[\chi(a-1, \lambda ), \chi(a-1, \lambda)] \\
&=2 i  \Big( \operatorname{Im} \lambda \sum_{t \in \mathcal{I}_N}  \chi^*(t, \lambda ) \chi(t, \lambda ) - \operatorname{Im} m  \Big),
\end{aligned}
$$
which, together with \eqref{1351352},  \eqref{38}, and \eqref{1351351}, implies that  $m $ is on the circle $C_{N}$ if and only if
the first  formula of \eqref{310} holds, and
 $m$ is inside the circle $C_{N}$ if and only if
the second  formula of \eqref{310} holds.
This completes the proof.
\end{proof}

\begin{corollary}\label{inside}
 If $N<N^{\prime} $, then $C_{N^{\prime }}$ is inside the circle $C_{N}$.
\end{corollary}

\begin{proof}
Let $m \in C_{N^{\prime }}$. Then, it is evident that\vspace{-2mm}
\begin{equation}\label{303146}
\sum_{t \in \mathcal{I}_N} \chi^*(t, \lambda ) \chi(t, \lambda ) < \sum_{t \in \mathcal{I}_{N'}} \chi^*(t, \lambda ) \chi(t, \lambda ) =\frac{\operatorname{Im} m }{\operatorname{Im} \lambda}.
\end{equation}
Therefore,  $m$ is inside   $C_{N}$ by Theorem \ref{mainlemma} which implies that $C_{N^{\prime }}$ is inside  $C_{N}$. This completes the proof.
\end{proof}

By  Corollary \ref{inside}, the sequence of circles $\{C_{N}\}$ converges   as  $N\rightarrow +\infty$. The limiting set is either a circle or a point.
Correspondingly, the classification of $\mathcal{{L}}$ or $(1.1_\lambda)$ can be given as follows.

\begin{definition}\label{mainlemma1}
If $\{C_{N}\}$ converges to a circle, then $\mathcal{{L}}$ or $(1.1_\lambda)$ is  called  to be in the limit circle case (LCC) at $t=\infty$,
and  if $\{C_{N}\}$ converges to a point,  then $\mathcal{{L}}$ or $(1.1_\lambda)$ is  called  to be in the limit point case (LPC) at $t=\infty$.
\end{definition}

\subsection{\normalsize\bf Relationships between square summable solutions and the      classification}
\ \ \ \ \ In this subsection, we shall derive an equivalent characterization  of the classification in terms of
the number of linearly independent solutions of $(1.1_\lambda)$ in $l^{2}(\mathcal{I})$. Here, we remark that solutions of $(1.1_\lambda)$ in $l^{2}(\mathcal{I})$
are also called square summable solutions of $(1.1_\lambda)$. The following is the main result of this section:

\begin{theorem}\label{LCC}
If there exists  $ \lambda_0 \in   {\Omega}^{\prime}(\mathcal{M},d) $ such that  $({1.1}_{\lambda_0})$  has two  linearly independent solutions
in $l^{2}(\mathcal{I}),$ then $\mathcal{{L}}$ or $(1.1_\lambda)$ is in the  LCC  at $ t=\infty$.
Otherwise, $\mathcal{{L}}$ or $(1.1_\lambda)$ is in the  LPC  at $ t=\infty$.
\end{theorem}

Before proving Theorem \ref{LCC}, we need derive three results in what follows.
First, we can get the following result:
\begin{lemma}\label{weyl}
If $\mathcal{{L}}$ is in the LCC at $t=\infty$, then $(1.1_\lambda)$ with $\lambda \in \mathbb{C} \backslash \mathbb{R}$ has exactly two linearly independent solutions in $l^{2}(\mathcal{I}),$ and
if $\mathcal{{L}}$ is in the LPC at $t=\infty$, then $(1.1_\lambda)$ with $\lambda \in \mathbb{C} \backslash \mathbb{R}$ has exactly one linearly independent solutions in $l^{2}(\mathcal{I})$.
\end{lemma}

\begin{proof}
If $\mathcal{{L}}$ is in the LCC at $t=\infty$, then  $\{C_{N}\}$ converges to a circle.  We take a point of this  circle  as $m$.
 If $\mathcal{{L}}$ is in the LPC at $t=\infty$, then $\{C_{N}\}$ converges to a point.  In this case, we take this point  as $m.$
 Then  $m\in C_{N}$ for  all $N>a$ by Corollary \ref{inside}. Let  $\chi(t, \lambda) $ be given by \eqref{03707} with this $m$. Then by the second formula of \eqref{310}, we have
$$ \sum_{t \in \mathcal{I}_N} \chi^*(t, \lambda ) \chi(t, \lambda )<\frac{\operatorname{Im} m }{\operatorname{Im} \lambda},~ ~  N>a, $$
which implies that  $\chi(\cdot, \lambda) \in l^{2}(\mathcal{I})$.

Furthermore, if $\{C_{N}\}$ converges to a circle, then   $\{r_{N}\}$ converges to
a positive  number. Then, from \eqref{34}, \eqref{1351352}, and \eqref{9312}, we get  $\psi(\cdot, \lambda) \in l^{2}(\mathcal{I}).$    Therefore $(1.1_\lambda)$ has two
linearly independent  solutions in $l^{2}(\mathcal{I})$ since   $\psi(t, \lambda)$ and $\chi\left(t, \lambda \right)$ are linearly independent on $\mathcal{I}$.
If  $\{C_{N}\}$ converges to a point, then  $\{r_{N}\}$ converges to $0$. From \eqref{1351352} and \eqref{9312}, we get  that $\psi(\cdot, \lambda) \notin l^{2}(\mathcal{I})$. Hence, $\chi(t, \lambda) $ is the only
linearly independent solution of $(1.1_\lambda)$ in $l^{2}(\mathcal{I})$.
This completes the proof.
\end{proof}

\begin{lemma}
Let $\lambda_{0}, \lambda \in {\Omega}^{\prime}(\mathcal{M},d) $ and
 $\varphi(t,\lambda_0)=(\varphi_1(t,\lambda_0),\varphi_2(t,\lambda_0))^{\mathrm{T}}$ and $\psi (t,\lambda_0) =(\psi_1(t,\lambda_0),\psi_2(t,\lambda_0))^{\mathrm{T}} $ be linearly independent solutions of $({1.1}_{\lambda_0})$.
Then for a solution  $z=(z_1,z_2)^{\mathrm{T}}$     of $({1.1}_{\lambda})$, there exist  two constants  ${k_1}$  and ${k_2}$ independently of $t$, i.e., only depending on $N\in \mathcal{I}$, $\lambda$, and  $\lambda_0$, such that for $t > N+2, $
\begin{equation}\label{3000}
 \begin{aligned}
z_1(t)&= {{k_1}}  \psi_{1}(t)+{{k_2} } \varphi_{1} (t)+\hspace{-1mm}\left(\lambda_0\hspace{-1mm} -\hspace{-1mm}\lambda\right)\hspace{-2mm}
\sum^{t-1}_{s=N+1 }\hspace{-2mm} \Big( \psi_{1}(t) \varphi^{\mathrm{T}}(s)-\varphi_{1}(t) \psi^{\mathrm{T}}(s)  \Big)z(s),
\end{aligned}
\end{equation}
and\vspace{-2mm}
\begin{equation}\label{30001}
 \begin{aligned}
z_2(t)= &\frac{  \lambda_0- \mathcal{M}(t)  }{ \lambda- \mathcal{M}(t) } \Big\{{{k_1}}\psi_{2}(t)+{{k_2}} \varphi_{2}(t)\\
&\ \ \ \ \ \ \ \ +\left(\lambda_0\hspace{-1mm}-\hspace{-1mm}\lambda \right)  \hspace{-2mm}\sum^{t}_{s=N+1 } \hspace{-2mm}\Big( \psi_{2}(t) \varphi^{\mathrm{T}}(s)-\varphi_{2}(t) \psi^{\mathrm{T}}(s)  \Big)z(s)
 \Big\}.
\end{aligned}
\end{equation}

\end{lemma}

\begin{proof}
Let $\lambda_{0}, \lambda \in {\Omega}^{\prime}(\mathcal{M},d) $  and  $z=(z_1,z_2)^{\mathrm{T}}$ be a solution    of $({1.1}_{\lambda})$, and for $t\in {\mathcal{I^{\prime}}}$ set\vspace{-2mm}
\begin{align*}
A(t)&:= \widetilde{p}(t, \lambda_0)\left[z_1(t+1,\lambda)\varphi_{1}(t,\lambda_0)- z_1(t,\lambda)\varphi_{1}(t+1,\lambda_0) \right],\\
B(t)&:= \widetilde{p}(t, \lambda_0)\left[  z_1(t+1,\lambda)\psi_{1}(t,\lambda_0)- z_1(t,\lambda)\psi_{1}(t+1,\lambda_0)\right].
\end{align*}
Then, we claim that  for   $t\in {\mathcal{I}},$  \vspace{-2mm}
\begin{align}
\label{11317} z_1(t,\lambda)&=  A(t-1)\psi_{1}(t,\lambda_0)-B(t-1)\varphi_{1}(t,\lambda_0), \\
\label{11318}z_2(t,\lambda)&= \frac{\lambda_{0}-d(t)}{\lambda-d(t)} \Big(A(t)\psi_{2}(t,\lambda_0)-B(t)\varphi_{2}(t,\lambda_0)\Big) .\vspace{-2mm}
\end{align}
In fact, by Lemma \ref{constant} and    \eqref{case}, one has
\begin{equation}\label{32563}
\widetilde{p}(t, \lambda_0)\left[ \psi_{1}(t+1,\lambda_0) \varphi_1(t,\lambda_0) - \psi_{1}(t,\lambda_0) \varphi_1(t+1, \lambda_0)\right]=1, ~ t\in {\mathcal{I^{\prime}}},
\end{equation}
and  thus, by the definitions of $A(t)$ and $B(t)$,  it is easy to verify that
\begin{equation}\label{34653}
z_1(t,\lambda)= A(t)\psi_{1}(t,\lambda_0)-B(t)\varphi_{1}(t,\lambda_0), ~ t\in {\mathcal{I^{\prime}}}.
\end{equation}
Since   $z$ is a solution of $({1.1}_{\lambda})$, we have  \eqref{002} and  \eqref{004} hold for $z$.
From \eqref{002} and  \eqref{004} for $z$ and by \eqref{003},   it can be derived  that
\begin{equation}\label{3150}
\begin{array}{ll}\vspace{2mm}
&- \nabla \left( \widetilde{p}(t, \lambda_0) \Delta z_{1}(t,\lambda)\right) +\left(\tilde{q}\left(t, \lambda_{0}\right)-\lambda_{0}\right) z_{1}(t,\lambda)\vspace{2mm} \\
=&\hspace{-3mm} \nabla \left[   \left(\widetilde{p}(t, \lambda)-\widetilde{p}(t, \lambda_0) \right) \Delta z_1 (t,\lambda) \right]
+\left(\tilde{q}\left(t, \lambda_{0}\right)-\tilde{q}\left(t, \lambda \right)\right) z_{1}(t,\lambda)+(\lambda-\lambda_0)z_1(t,\lambda) \vspace{3mm}\\
=& \hspace{-3mm}\left(\lambda_0-\lambda\right)\left\{ \nabla\left[\displaystyle\frac{c(t)}{\lambda_0-d(t)}z_2(t, \lambda)-\displaystyle\frac{h(t)c(t)z_1(t+1, \lambda)}{(\lambda_0-d(t))(\lambda-d(t))}\right]\right. \vspace{3mm}\\
&\hspace{-3mm}\left. -\displaystyle\frac{h^2(t)z_1(t, \lambda)}{(\lambda_0-d(t))(\lambda-d(t))}+\nabla\left(\displaystyle\frac{h(t)c(t)}{(\lambda_0-d(t))(\lambda-d(t))}\right)z_1(t, \lambda)-z_1(t, \lambda) \right\}  \vspace{3mm}\\
=&\hspace{-3mm} \left(\lambda_0-\lambda\right)\left[\nabla\left(\displaystyle\frac{c(t)}{\lambda_0-d(t)} z_2(t, \lambda)\right)- \displaystyle\frac{h(t)}{\lambda_0 -d(t) }z_2(t, \lambda)-z_1(t, \lambda) \right], ~ t\in \mathcal{I}. \\
\end{array}
\end{equation}
In addition,  \eqref{002} holds for $\varphi_1$, i.e.,
\begin{equation}\label{3151}
\begin{aligned}
&- \nabla \left( \widetilde{p}(t, \lambda_0) \Delta \varphi_{1}(t, \lambda_0)\right) +\left(\tilde{q}\left(t, \lambda_{0}\right)-\lambda_{0}\right) \varphi_{1}(t, \lambda_0)=0, ~ t\in \mathcal{I}.\\
\end{aligned}
\end{equation}
Here, we remark that similarly to those in \eqref{3000} and \eqref{30001}, we omit $\lambda$ and $\lambda_0$ for
simplicity, e.g., write $\varphi_{1}(t, \lambda_0)$ as $\varphi_{1}(t)$ and $z_{1}(t,\lambda)$ as $z_{1}(t)$, in what follows.
Multiplying both sides of \eqref{3150} by $-\varphi_{1}(t)$  and    \eqref{3151} by $z_{1}(t)$,  and adding them give  that
\begin{equation}\label{3180}
\nabla A(t) =-\left(\lambda_{0}-\lambda\right)\left[\nabla\left(\frac{c(t)}{\lambda_0-d(t)} z_2(t)\right) - \frac{h(t)}{\lambda_0 -d(t) }z_2(t)-z_1(t) \right]\varphi_{1}(t), ~ t\in \mathcal{I}.
\end{equation}
With a similar argument to that of \eqref{3180},
we have
\begin{equation}\label{3181}
\nabla B(t) =-\left(\lambda_{0}-\lambda \right)\left[ \nabla\left(\frac{c(t)}{\lambda_0-d(t)} z_2(t)\right) - \frac{h(t)}{\lambda_0 -d(t) }z_2(t)-z_1(t) \right]\psi_{1}(t), ~ t\in \mathcal{I}.
\end{equation}
Multiplying both sides of   \eqref{3180} by $-\psi_{1}(t)$  and \eqref{3181} by  $\varphi_{1}(t)$,  and adding them give that
\begin{equation}\label{11328}
A(t)\psi_{1}(t)-B(t)\varphi_{1}(t)=A(t-1)\psi_{1}(t)-B(t-1)\varphi_{1}(t), ~ t\in \mathcal{I},
\end{equation}
which   yields that   \eqref{11317} holds by \eqref{34653}.
From  \eqref{11317}  and  \eqref{11328}, it can be obtained  that \vspace{-2mm}
\begin{equation}\label{1132880}
 \Delta z_1(t)=A(t)\Delta\psi_{1}(t)-B(t)\Delta \varphi_{1}(t), ~ t\in {\mathcal{I}}.\vspace{-2mm}
\end{equation}
 Then, from \eqref{004} for $z$, $\varphi$, and $\psi$, respectively, \eqref{34653}, and \eqref{1132880}, we get  that   \eqref{11318} holds.

Next,  for $t> N+2$,  summing up \eqref{3180} from $N+1$ to $t$ gives
\begin{equation}\label{3202}
\begin{aligned}
A(t)&=A(N)-\left(\lambda_0-\lambda \right)\hspace{-1mm}
\sum^{t}_{s=N+1} \hspace{-1mm}\left[   \nabla\left(\hspace{-1mm}\frac{c(s)}{\lambda_0-d(s)} z_2(s)\hspace{-1mm}\right) - \frac{h(s)}{\lambda_0 -d(s) }z_2(s) -z_1(s)  \right] \varphi_1(s)\\
&=A(N)-\left(\lambda_0-\lambda \right) \Big[L(t)-L(N)- \sum^{t}_{s=N+1}\big(z_2(s)\varphi_{2}(s)+z_1(s)\varphi_{1}(s)\big)\Big],\\
&=A(N)-\left(\lambda_0-\lambda \right) \Big(L(t)-L(N)-   \sum^{t}_{s=N+1} \varphi^{\mathrm{T}}(s) z(s) \Big),\\
\end{aligned}
\end{equation}
where $L(t):=\dfrac{c(t)}{\lambda_0 -d(t)}z_{2}(t)\varphi_{1}(t+1).$
Similarly, for $t> N+2$, we get \vspace{-2mm}
\begin{equation}\label{39468}
B(t) =B(N)-\left(\lambda_0-\lambda \right)\Big(K(t)-K(N)- \sum^{t}_{s=N+1} \psi^{\mathrm{T}}(s) z(s)     \Big),\vspace{-4mm}
\end{equation}
where  $K(t):=\dfrac{c(t)}{\lambda_0 -d(t)}z_{2}(t)\psi_{1}(t+1).$ On the other hand, we get  that
\begin{equation}\label{323343}
L(t-1)\psi_{1}(t)-K(t-1)\varphi_{1}(t)=0,\ t> N+2,\vspace{-2mm}
\end{equation}
and    from  \eqref{004} and \eqref{32563} that
\begin{equation}\label{14328}
L(t)\psi_{2}(t)-K(t)\varphi_{2}(t)=\frac{c^{2}(t)-h(t)c(t) }{ {(\lambda_0-d(t))}^{2} \widetilde{p}(t,\lambda_0) }  z_{2}(t), \ t> N+2.
\end{equation}
Then,  inserting \eqref{3202}-\eqref{323343} into  \eqref{11317}, we get that there exist  two constants  ${k_1}$  and ${k_2}$ depending on $N$, $\lambda$, and  $\lambda_0$,
such that \eqref{3000} holds for $t> N+2$.
Similarly,  inserting \eqref{3202}, \eqref{39468}, and  \eqref{14328} into    \eqref{11318},  we get
\begin{equation}\label{143330}
\begin{aligned}
z_2(t)= &\frac{  \lambda_0- d(t)  }{ \lambda- d(t) } \Big\{{{k_1}}\psi_{2}(t)+{{k_2}} \varphi_{2}(t)
 +\left(\lambda_0\hspace{-1mm}-\hspace{-1mm}\lambda \right)   \Big[\frac{h(t)c(t)-c^{2}(t) }{ {(\lambda_0-d(t))}^{2} \widetilde{p}(t,\lambda_0) }  z_{2}(t)\\
 &\qquad \qquad  \ +\sum^{t}_{s=N+1 } \hspace{-2mm}\Big( \psi_{2}(t) \varphi^{\mathrm{T}}(s)-\varphi_{2}(t) \psi^{\mathrm{T}}(s)  \Big)z(s)
 \Big] \Big\}, \ t> N+2,
\end{aligned}
\end{equation}
where ${k_1}$  and ${k_2}$ are the same as those in \eqref{3000}.
It can be easily verified that  $$1-\frac{ (\lambda_0 - \lambda  )    (h(t)c(t)-c^{2}(t) )  }{ (\lambda- d(t)) {(\lambda_0-d(t))} \widetilde{p}(t,\lambda_0)}     = \frac{  \lambda_0- d(t)  }{ \lambda- d(t) } \frac{  \lambda - \mathcal{M}(t)  }{ \lambda_0- \mathcal{M}(t) },  $$
which, together with  \eqref{143330},
yields that  \eqref{30001} holds for $t> N+2$. This completes the proof.
\end{proof}

\begin{lemma}\label{31}
If there exists $\lambda_{0}  \in {\Omega}^{\prime}(\mathcal{M},d) $ such that $({1.1}_{\lambda_0})$  has two  linearly independent solutions in $l^{2}(\mathcal{I})$, then it is true for all
$ \lambda \in {\Omega}^{\prime}(\mathcal{M},d) .$
\end{lemma}

\begin{proof}
Suppose that
$\varphi =(\varphi_1 ,\varphi_2 )^{\mathrm{T}}$ and $\psi =(\psi_1 ,\psi_2 )^{\mathrm{T}} $ are linearly independent solutions of $({1.1}_{\lambda_0})$ in $l^{2}(\mathcal{I})$ for some $\lambda_{0} \in {\Omega}^{\prime}(\mathcal{M},d) $ by the assumption. If  $z=(z_1,z_2)^{\mathrm{T}}$  is a solution    of $({1.1}_{\lambda})$ with $\lambda \in {\Omega}^{\prime}(\mathcal{M},d)$, then \eqref{3000} and \eqref{30001} hold.
Applying the Cauchy-Schwarz inequality,  we get from  \eqref{3000} that
$$
\left|z_{1}(t)\right| \leq  (|\varphi_1 (t)|+|\psi_1 (t)|)\Big[\left(|k_{1}|+\left|k_{2}\right|\right) +\left|\lambda-\lambda_{0}\right|(\|\varphi\|+\|\psi\|) \|z\|_{N+1}^{t-1}\Big], \vspace{-2mm}
$$
where   $\|z  \|_{N+1}^{t-1} ={ \big(  \sum_{s=N+1}^{t-1}  { z^{*} (s)  z (s) }  \big) }^{\frac{1}{2}} .$
Then, from the above relation  we get that there exists $K_1>0$ such that
for $\tau>N_0>N+2$,
\begin{equation}\label{32763}
\|z_{1}\|_{N_{0}}^{\tau } \leq K_1\gamma_{N_0} \Big[ 1+   \left|\lambda-\lambda_{0}\right|  \left(\|z\|^{N_{0}}_{N  }+ \|z\|_{N_{0}}^{\tau  }  \right)     \Big], \vspace{-2mm}
\end{equation}
where   $\|z_1 \|_{N_0}^{\tau} ={ \big(  \sum^{\tau}_{s=N_0}  {| z_1 (s)|}^{2} \big) }^{\frac{1}{2}} $ and
$$\gamma_{N_0}={ \left(  \sum^{\infty}_{t=N_0}  {\varphi^*(t)\varphi(t)} \right) }^{\frac{1}{2}}+{ \left(  \sum^{\infty}_{t=N_0}  {\psi^*(t)\psi(t)} \right) }^{\frac{1}{2}}.$$

Furthermore, since  $\lambda \in {\Omega}^{\prime}(\mathcal{M},d),  $  we get $ \inf_{t\in \mathcal{I^{\prime}}} \Big|\lambda -\mathcal{M}(t)   \Big| >0,$ and thus
\begin{equation}\label{1401}
\Big|\frac{  \lambda_0- \mathcal{M}(t)  }{ \lambda- \mathcal{M}(t) }\Big|=1+\frac{ | \lambda_0-\lambda | }{ \inf_{t\in \mathcal{I^{\prime}}} \Big|\lambda -\mathcal{M}(t)   \Big|}< \infty.\vspace{-2mm}
\end{equation}
Similarly, we can get from \eqref{30001} and \eqref{1401} that  there exists $K_2>0$  such that for $\tau>N_0>N+2$,
\begin{equation}\label{3277}
\|z_{2}\|_{N_{0}}^\tau  \leq K_2  {\gamma_{N_0} }\Big[ 1+   \left|\lambda-\lambda_{0}\right|\left( \|z\|^{N_{0}}_{N  }  +\|z\|_{N_{0}}^\tau\right )       \Big]. \vspace{-2mm}
\end{equation}
Since $\varphi,\psi \in l^{2}(\mathcal{I})$,  we have  $\gamma_{N_0}\rightarrow 0 $ as $N_0 \rightarrow \infty. $  Thus, letting $K_0:=\max\{K_1,K_2\}$, we can choose  sufficiently large $N_0 $ satisfying $  K_0 \left|\lambda-\lambda_0 \right| {\gamma_{N_0 } }\leq \frac{1}{4} $.
Then, from \eqref{32763} and \eqref{3277}, we have\vspace{-2mm}
\begin{equation}\label{3288}
\|z_{j}\|_{N_0}^{\tau } \leq  K_0 \gamma_{N_0}\left(1 +|\lambda-\lambda_0|\|z\|^{N_{0}}_{N  }\right ) + \frac{1}{4}\|z\|_{N_0}^{\tau },\ j=1,2, \vspace{-2mm}
\end{equation}
 which implies that $z \in l^{2}(\mathcal{I}). $
This completes the proof.
\end{proof}

Finally, we prove Theorem \ref{LCC}.

\noindent{\bf Proof of Theorem \ref{LCC}.} It is noted that $\mathbb{C} \setminus\mathbb{R}\subset{\Omega}^{\prime}(\mathcal{M},d) $ since the coefficients of $({1.1}_{\lambda})$ are real-valued.
Suppose that there exists  $ \lambda_0 \in   {\Omega}^{\prime}(\mathcal{M},d) $ such that  $({1.1}_{\lambda_0})$  has two  linearly independent solutions
in $l^{2}(\mathcal{I})$. Then it is true for all $\lambda \in \mathbb{C} \setminus \mathbb{R}$ by Lemma 3.4, which implies that $\mathcal{{L}}$ or $(1.1_\lambda)$ is in the  LCC  at $ t=\infty$ by Lemma 3.2.
Otherwise, there exists at most one linearly independent solutions of $(1.1_\lambda)$ in $l^{2}(\mathcal{I})$ for each $\lambda\in{\Omega}^{\prime}(\mathcal{M},d)$.
Then $\mathcal{{L}}$ or $(1.1_\lambda)$ is in the  LPC  at $ t=\infty$ by Lemma 3.2. This completes the proof.

\begin{remark} {\rm It is noted that equations \eqref{00} contain \eqref{3}  as their special case. Therefore, Definition 3.1 and Theorem 3.2 are also
applied to \eqref{3} which is useful in the next section. In fact,
Jirari \cite{Jirari} has considered  singular Sturm-Liouville difference  equations $\tau y_1=\lambda wy_1$ on $\mathcal{I}$, where $\tau$ is given by \eqref{3} and $w>0$ is a weight function.
Similar classification  and result as given by  Definition 3.1  and        Theorem 3.2   were obtained for $\tau y_1=\lambda wy_1$ on $\mathcal{I}$ in \cite{Jirari}.
Definition 3.1 and Theorem 3.2 for  \eqref{3} are their special case of $w\equiv1$.}
\end{remark}

\section{{\large{\bf On perturbations of matrix
difference equations  }}}{}\baselineskip 20pt
\ \ \ The difference expression $\mathcal{{L}}$ can be interpreted as
\begin{equation}\label{15410}
\mathcal{{L}}=\mathcal{{L}}^{(0)}+\mathcal{{L}}^{(1)}  ,
\end{equation}
where
$$
{\mathcal{L}}^{(0)}\left(\begin{array}{l}
y_1  \\
y_2 \end{array}\right)(t):=\left\{\begin{array}{l}{\rm diag}\{- \nabla p  \Delta+q ,d\}\left(\begin{array}{l}
y_1  \\
y_2
\end{array}\right)(t),~ t\in \mathcal{I},\vspace{2mm}\\
{\rm diag}\{0 ,d\}\left(\begin{array}{l}
y_1  \\
y_2
\end{array}\right)(t),~ t=a-1,  \end{array}\right.$$
$$
{\mathcal{L}}^{(1)}\left(\begin{array}{l}
y_1  \\
y_2 \end{array}\right)(t):=\left\{\begin{array}{l}\left(\begin{array}{cc}
0 & -\nabla c +h  \\
c  \Delta+h  & 0
\end{array}\right)\left(\begin{array}{l}
y_1  \\
y_2
\end{array}\right)(t),~ t\in \mathcal{I},\vspace{2mm}\\
\left(\begin{array}{cc}
0 &0  \\
c  \Delta+h  & 0
\end{array}\right)\left(\begin{array}{l}
y_1  \\
y_2
\end{array}\right)(t),~ t=a-1.  \end{array}\right.$$
If  $c$ and $h$ are   bounded on ${\mathcal{I^{\prime}}}$, then it can be verified that the  limit point or limit circle type of $\mathcal{{L}}^{(0)}$ is equal to that of $\mathcal{{L}}$. A natural  question  is whether the limit type is invariant   if $c$ or $h$ is   unbounded on ${\mathcal{I^{\prime}}}$.
Hassi, M\"{o}ller, and Snoo considered equation \eqref{2222} on the interval $[0,\infty)$ with real-valued coefficients $p$, $q$, $c$, $h$, and $d$. It was shown that the limit type of  $\mathfrak{L}^{(0)}$ is different from that of $\mathfrak{L}$ in general when    $c$ or $h$ is   unbounded on $[0,\infty)$ by \cite[Examples 6.3 and 6.4]{Hassi},
where $\mathfrak{L}^{(0)}$ and   $\mathfrak{L}^{(1)}$ are given by
$$
\mathfrak{L}^{(0)}=\left(\begin{array}{cc}
- \displaystyle D p \displaystyle D+q   & 0 \vspace{2mm}\\
0 & d
\end{array}\right),\ \
\mathfrak{L}^{(1)}=\left(\begin{array}{cc}
0 & \displaystyle D c +h  \\
c D+h  & 0
\end{array}\right).
$$
Here, we shall show that this is also true for  $\mathcal{{L}}^{(0)}$ and $\mathcal{{L}}$ when   $\mathcal{{L}}^{(1)}$ is  given here  with    $c$ or  $h$  being    unbounded on $\mathcal{I}$ by two examples.
The   first example shows that  $\mathcal{{L}}^{(0)}$ is  in the  LCC at $t= \infty$ while $\mathcal{{L}}$ is in the LPC at $t= \infty$, and     the second   one shows that   $\mathcal{{L}}^{(0)}$ is in the  LPC  at $t= \infty$ while $\mathcal{{L}}$ is  in the LCC  at $t= \infty$.

\begin{example}
{\rm  Consider $\mathcal{{L}}^{(0)}$  with $p(t)=-4^{t}$, $q(t)=4^{t}$,  and  $d=1 $ for $t\in {\mathcal{I^{\prime}}}=\{t\}_{t=-1}^{+\infty} $.
 It is evident that $\widetilde{p}(t, \lambda)$ and  $  \tilde{q}(t, \lambda)$  associated with  $\mathcal{{L}}^{(0)}(y)=\lambda y$
  are given by\vspace{-2mm}
$$ \widetilde{p}(t, \lambda)= -4^{t},t\in {\mathcal{I^{\prime}}};~  \tilde{q}(t, \lambda)=4^{t},\ t\in \mathcal{I}=\{t\}_{t=0}^{+\infty}.\vspace{-2mm}$$
Therefore,  the corresponding equation \eqref{002} becomes as\vspace{-2mm}
\begin{equation}\label{1501}
\tilde{\tau}(y_1)(t):=\nabla \big(4^{t} \Delta y_1(t)\big)+4^{t}y_1(t)=\lambda y_1(t),~ t\in\mathcal{I}.\vspace{-2mm}
\end{equation}
By {\cite[Example 3.2]{Chen}}, \eqref{1501} is in the   LCC at  $t= \infty$.
Then all its solutions $y_1 $ satisfy $ \sum\limits_{t=0}^{\infty}|y_1(t)|^{2}<\infty.$ In addition, we get from  \eqref{004}  that $y_2=0 $ with $\lambda\neq 1$ since  $c=h=0$ on ${\mathcal{I^{\prime}}}$ and $d(t)\neq0 $ for $t\in {\mathcal{I^{\prime}}},$  which implies that all solutions of $\mathcal{{L}}^{(0)}(y)=\lambda y $ with $\lambda\neq 1$ are in $l^{2}(\mathcal{I})$.  Therefore, $\mathcal{{L}}^{(0)}$ is in the  LCC at  $t= \infty$.

Now, take $\mathcal{{L}}^{(1)}$  with  $h(t)= 2^{t}+2^{-t}$ and $c(t)=0,
$ $t\in {\mathcal{I^{\prime}}}$. Then
$ \mathcal{M}=d=1$ on ${\mathcal{I^{\prime}}}$, and hence
${\Omega}^{\prime}(\mathcal{M},d)=\mathbb{C}\setminus\{1\}$. For $\mathcal{{L}}(y)=\lambda y$  with $\lambda  \in \mathbb{C}\setminus \{1\}$ and
$\mathcal{{L}}$ given by  \eqref{15410},
$\widetilde{p}(t, \lambda)$ and  $  \tilde{q}(t, \lambda)$   are given by\vspace{-2mm}
$$
\widetilde{p}(t, \lambda )=-4^{t}, t\in {\mathcal{I^{\prime}}};~
\tilde{q}(t, \lambda )=4^{t}+\frac{4^{t}+ 4^{-t}+2}{\lambda-1}, ~  t\in\mathcal{I}.\vspace{-2mm}
$$
Therefore, the corresponding equation \eqref{002}  with $ \lambda=0$  becomes as\vspace{-2mm}
\begin{equation}\label{1503}
 \nabla \big( 4^{ t}  \Delta y_1(t) \big)-(4^{-t}+2)y_1(t)=0,\ t\in\mathcal{I}.\vspace{-2mm}
\end{equation}
By  \cite[Therorem 3.11.6]{Jirari},  \eqref{1503}  has  a solution $y_1$
satisfying  $ \sum\limits_{t=0}^{\infty}|y_1(t)|^{2}=\infty.$
Let $y =(y_1,y_2 )^{\mathrm{T}} $ with  $y_2$  given by \eqref{004} with $ \lambda=0$. Then
$y$ is a solution of   $\mathcal{{L}}(y)=0  .$ Clearly   $y  \notin l^{2}(\mathcal{I}).$
Hence,    $\mathcal{{L}}$  is  in the  LPC  at $ t=\infty$.
}
\end{example}

\begin{example}
{\rm Consider $\mathcal{{L}}^{(0)}$  with    $p=1$, $q(t)=4^{t}$, and $d(t)=4^{t}$  for $t\in {\mathcal{I^{\prime}}}=\{t\}_{t=-1}^{+\infty}$.
 It is evident that $\widetilde{p}(t, \lambda)$ and  $  \tilde{q}(t, \lambda)$  associated with  $\mathcal{{L}}^{(0)}(y)=\lambda y$
  are given by\vspace{-2mm}
$$ \widetilde{p}(t, \lambda)=1,t\in {\mathcal{I^{\prime}}};~  \tilde{q}(t, \lambda)=4^{t}, \ t\in \mathcal{I}=\{t\}_{t=0}^{+\infty}, \vspace{-2mm}$$
Therefore, the corresponding equation \eqref{002} becomes as \vspace{-2mm}
\begin{equation}\label{1601}
\nabla (\Delta y_1(t))+4^{t}y_1(t)=\lambda y_1(t),\ t\in\mathcal{I}. \vspace{-2mm}
\end{equation}
By {\rm \cite[Corollary 3.1]{Chen}}, \eqref{1601} has  a solution $y_1$
satisfying $ \sum\limits_{t=0}^{\infty}|y_1(t)|^{2}=\infty.$  Then, $\mathcal{{L}}^{(0)}(y) =\lambda y $ for $\lambda$ with ${\rm Im}\lambda\neq0$ has  a solution $y  \notin l^{2}(\mathcal{I}).$
Therefore, $\mathcal{{L}}^{(0)}$  is  in the  LPC  at $ t=\infty$.

Now, take $\mathcal{{L}}^{(1)}$  with  $c(t)= \sqrt{4^{2t}+4^{t}}$ and $h(t)=0,~t\in {\mathcal{I^{\prime}}}$.   Then $\mathcal{M}(t)=-4^{2t},t\in{\mathcal{I^{\prime}}}, $
and  then $ \sigma(\mathcal{M}) \cup \sigma(d)=  \{-4^{2t}, 4^{t}: t\in {\mathcal{I^{\prime}}} \}. $
Thus ${\Omega}^{\prime}(\mathcal{M},d)=\mathbb{C}\setminus \{-4^{2t},4^{t}: t\in {\mathcal{I^{\prime}}} \}$. For $\mathcal{{L}}(y)=\lambda y$  with $\lambda \in {\Omega}^{\prime}(\mathcal{M},d)$ and
$\mathcal{{L}}$ given by  \eqref{15410},
$\widetilde{p}(t, \lambda)$ and  $  \tilde{q}(t, \lambda)$   are given by
$$
\widetilde{p}(t, \lambda)=1+\frac{4^{2t}+4^{t}}{\lambda-4^{t}},~t\in {\mathcal{I^{\prime}}};~  \tilde{q}(t, \lambda)=4^{t}, ~t  \in \mathcal{I}.
$$
Note that $\lambda=0\in {\Omega}^{\prime}(\mathcal{M},d)$. Then take $ \lambda=0$ and the corresponding equation  \eqref{002}   becomes  as
$\tilde{\tau}(y_1)(t)=0,~t\in \mathcal{I}$, where $\tilde{\tau}$ is given by  \eqref{1501}.
Then $\tilde{\tau}$ is  in the LCC at   $ t=\infty$, which implies that  all   solutions   $y_1(t)$  of $\tilde{\tau}(y_1)(t)=0 $ satisfy $ \sum\limits_{t=0}^{\infty}|y_1(t)|^{2}<\infty,$
In addition,   \eqref{004}   becomes  as
\begin{equation}\label{1604}
y_2(t)=\frac{\sqrt{4^{2t}+4^{t}}}{ -4^{t}} \Delta y_1(t),~ t  \in {\mathcal{I^{\prime}}},
\end{equation}
which, together with  $ \sum\limits_{t=0}^{\infty}|y_1(t)|^{2}<\infty,$  yields that   $ \sum\limits_{t=0}^{\infty}|y_2(t)|^{2}<\infty.$
Let $y =(y_1,y_2 )^{\mathrm{T}} $ with  $y_2$ given by \eqref{1604}. Then
$y$ is a solution of   $ \mathcal{{L}}(y)=0  .$
Clearly   $y =(y_1,y_2 )^{\mathrm{T}}   \in l^{2}(\mathcal{I}).$
Hence,  $\mathcal{{L}}$  is in the  LCC  at $ t=\infty$.
}
\end{example}

\section{{\large{\bf  Limit point criteria    }}}{}\baselineskip 20pt
\ \ \  \ In this section, we shall establish two criteria of the limit point case for   $\mathcal{{L}}$ in terms of its coefficients which extend the existing results for Sturm-Liouville differential and difference expressions to
matrix
difference expressions $\mathcal{{L}}$.

\begin{theorem}\label{5252}
If there exist  $N \in \mathcal{I}$  and $K$ such that
$\left|\displaystyle\frac{c(t)}{p(t)} \right|\leq K$ for $t> N,$  and $
\sum\limits_{t =N }^{\infty} \dfrac{1}{|p(t)|}=\infty
$, then   $\mathcal{{L}}$ is in the  LPC  at $ t=\infty$.
\end{theorem}

\begin{proof}
Suppose on the contrary that  $\mathcal{{L}}$ is  in the  LCC  at $ t=\infty$.  Then,  $\varphi= (\varphi_{1}, \varphi_{2} )^{\mathrm{T}}  $  and $\psi=(\psi_{1}, \psi_{2})^{\mathrm{T}}$
given in Section 3  satisfying   \eqref{case} are linearly independent solutions of  \eqref{00} in $l^{2}(\mathcal{I})$.
Further, by Lemma \ref{constant} and  \eqref{case}, one has $[\varphi(t), \overline{\psi}(t)]=1$ on $  {\mathcal{I^{\prime}}}$,
which, together with  \eqref{0041}, yields  that
$$
p(t)\left[{\varphi}_{1}(t+1) {\psi}_{1}(t )- {\varphi}_{1}(t) {\psi}_{1}(t+1)\right] +c(t)\left[{ \psi}_{1}(t+1) {\varphi}_{2}(t)- {\psi}_{2}(t) {\varphi}_{1}(t+1)\right]=-1.
$$
Since  $\left|\displaystyle\frac{c(t)}{p(t)} \right|\leq K$ for $t> N,$ it follows that for $t> N,$\vspace{-2mm}
\begin{equation}\label{4311}
\begin{gathered}
\left| {\psi}_{1}(t+1)\right|\left( {\varphi}_{1}(t)+ K\left| {\varphi}_{2}(t)\right| \right)
+\left| {\varphi}_{1}(t+1)\right|\left( {\psi}_{1}(t)+ K\left| {\psi}_{2}(t)\right| \right)
\geq \frac{1}{|p(t)|} .
\end{gathered}\vspace{-2mm}
\end{equation}
By the Cauchy's inequality, the left-hand side of \eqref{4311} is summable, which contradicts  to  $
\sum\limits_{t =N }^{\infty} \displaystyle\frac{1}{|p(t)|}=\infty .$  Therefore, $\mathcal{{L}}$ is  in the LPC  at $ t=\infty$. This completes the proof.
\end{proof}

\begin{remark}
\quad
{\rm
\begin{itemize}
\item [{\rm (1)}] It is noted that the criterion given by Theorem 5.1 only depending on the coefficients $p(t)$ and $c(t)$ for $ t > N .$\vspace{-2mm}
\item [{\rm (2)}] By Theorem 5.1,
${\mathcal{L}}^{(0)}$ is in the  LPC  at $ t=\infty$ if $
\sum^{\infty}\limits_{t=N } \dfrac{1}{|p(t)|}=\infty$, and this limit point case is invariant under the perturbation  $\mathcal{{L}}^{(1)}$ under   condition $\left|\displaystyle\frac{c(t)}{p(t)} \right|\leq K , t > N$.
\vspace{-2mm}
\item [{\rm (2)}]
Hinton and  Lewis \cite{Hinton} considered equation $\tau y_1=\lambda wy_1$, i.e.,\vspace{-2mm}
\begin{equation}\label{1701}
-\nabla (p(t) \Delta y_1 (t))+q(t)y_1 (t)=\lambda w(t) y_1(t), ~ t\in \mathcal{I},\vspace{-2mm}
\end{equation}
where $ w(t)>0$, $p$ and $q$ are real-valued on $\mathcal{I}$. By \cite[Theorem 10]{Hinton}, if $$ \sum_{t \in \mathcal{I} } \dfrac{ {(w(t) w(t+1))} ^{\frac{1}{2}}  }{|p(t)|}=\infty,$$ then \eqref{1701}  is in the  LPC  at $ t=\infty$. It is noted that \eqref{00} contains  \eqref{1701} as its special case of $h(t)=c(t)\equiv0$  and $w(t)\equiv1$ on ${\mathcal{I^{\prime}}} $ and   $d(t)\neq0$ for $t\in {\mathcal{I^{\prime}}}.$
Then, Theorem 5.1 is a generalization of  \cite[Theorem 10]{Hinton} with $w(t)\equiv1$.

\end{itemize}
}
\end{remark}

\begin{theorem}\label{51512}
If  $p(t)>0,~t\in \mathcal{I}$ and there exist $N\in \mathcal{I}$, a sequence of positive numbers $\{M(t)\}_{t=N}^{\infty}$, and   positive constants $k_{j}, ~1\leq j\leq4, $ such that for all $t > N$,
\begin{align*}
& 1) \quad|c(t)|+|c(t-1)| \leq k_{1} M(t), \  |h(t)|  \leq k_{2} M(t),\vspace{-2mm}\\
& 2) \quad  q(t) \geq-k_{3} M(t),  \hspace{12cm}\vspace{-2mm}\\
& 3) \quad \frac{p^{\frac{1}{2}}(t-1)|\nabla M(t)|}{M^{\frac{1}{2}}(t) M(t-1)} \leq k_{4},\\
& 4)   \quad \sum_{t=N}^{\infty} \frac{1}{\left(p^{2}(t-1 )+c^{2}(t-1 )\right)^{\frac{1}{4}} M^{\frac{1}{2}}(t)}=\infty, \vspace{-2mm}
\end{align*}
then $\mathcal{{L}}$ is in the LPC  at $ t=\infty$.
\end{theorem}

\begin{proof}
Suppose that  $ y=(y_{1}, y_{2})^{\mathrm{T}}  $ is a solution  of  \eqref{00} with $\lambda=i$. Then, we have
\begin{equation}\label{4002}
-\nabla(p(t) \Delta y_1(t)) -\nabla(c(t) y_2(t))+h(t)  y_2(t) +(q(t) -i) y_1(t)  =0,~t\in \mathcal{I}.
\end{equation}
Multiplying  both side of \eqref{4002} by $ \displaystyle\frac{\overline{y}_1(t)}{M(t)}$ and with a simple calculation, we get that
\begin{equation}\label{4003}
\begin{gathered}
\nabla\left(\frac{p(t)(\Delta y_1(t)) \overline{y}_1(t)}{M(t)}\right)=\frac{p(t-1)|\nabla y_1(t)|^{2}}{M(t)}-\frac{p(t-1)(\nabla M(t))(\nabla y_1(t)) \overline{y}_1(t-1)}{M(t) M(t-1)} \\
+\frac{(q(t)-i)|y_1(t)|^{2}}{M(t)}+\frac{h(t) y_2(t) \overline{y}_1(t)}{M(t)}-\frac{\nabla(c(t)y_2(t)) \overline{y}_1(t)}{M(t)}.
\end{gathered}
\end{equation}
Summing up \eqref{4003} from $N$ to $t$ yields
\begin{equation}\label{4004}
\begin{gathered}
\frac{p(t)(\Delta y_1(t)) \overline{y}_1(t)}{M(t)}=G(t)-\sum_{s=N}^{t} \frac{p(s-1)(\nabla M(s))(\nabla y_1(s)) \overline{y}_1(s-1)}{M(s) M(s-1)} \\
+\sum_{s=N}^{t} \frac{(q(s)-i)|y_1(s)|^{2}}{M(s)} + \sum_{s=N}^{t} \frac{h(s) y_2(s) \overline{y}_1(s)}{M(s)}  - \sum_{s=N}^{t} \frac{\nabla(c(s)y_2(s)) \overline{y}_1(s)}{M(s)}+c_{0}
\end{gathered}
\end{equation}
where
$$
G(t)=\sum_{s=N}^{t} \frac{p(s-1)|\nabla y_1(s)|^{2}}{M(s)} \quad \text { and } \quad c_{0}=\frac{p(N-1)(\Delta y_1(N-1)) \overline{y}_1(N-1)}{M(N-1)}.
$$
Since $p(t)>0$ and $M(t)>0,$   $\lim\limits _{t \rightarrow \infty} G(t) $ exists which may be infinity.  Now suppose that $y=\left(y_1, y_2\right)^{\mathrm{T}} \in l^{2}(\mathcal{I}).$
We shall show that $\lim\limits _{t \rightarrow \infty} G(t)<\infty$  in this case.
By the assumptions 1)-3), the Cauchy's inequality,  $y \in l^{2}(\mathcal{I}),$ it follows from \eqref{4004}  that there exist $\tilde{k}_1, \tilde{k}_2\in \mathbb{R}$ such that for $t> N$,
\begin{equation}\label{4005}
\operatorname{Re}\left\{\frac{p(t)(\Delta y_1(t)) \overline{y}_1(t)}{M(t)}\right\} \geq G(t)-\tilde{k}_1 G^{1 / 2}(t)+ \tilde{k}_2.
\end{equation}
Assume  on the contrary that $\lim\limits _{t \rightarrow \infty} G(t)=\infty$. Then, \eqref{4005} yields that there exists a positive integer $N_{1} \geq N$ such that
\begin{equation}\label{411}
\operatorname{Re}\{(\Delta y_1(t)) \overline{y}_1(t)\}>0, \quad t > N_{1},
\end{equation}
i.e.,
\begin{equation}\label{412}
\frac{1}{2}( y_1(t+1) \overline{y}_1(t)+\overline{y}_1(t+1)  y_1(t))-\overline{y}_1(t)  y_1(t)>0, \quad t > N_{1} .
\end{equation}
It is obtained from \eqref{411} that $y_1(t) \neq 0$ for $t > N_{1}$. Therefore, \eqref{412} implies that
$$
\operatorname{Re}\left\{\frac{y_1(t+1)}{y_1(t)}\right\}>1, \quad t> N_{1}.
$$
Hence,  $ \sum\limits_{t=N}^{\infty}|y_1(t)|^{2}=\infty,$  which is contrary to assumption $y \in l^{2}(\mathcal{I})$. Therefore, we have $\lim\limits _{t \rightarrow \infty} G(t)<\infty$.

Now, let  $\varphi= (\varphi_{1}, \varphi_{2} )^{\mathrm{T}}  $  and $\psi=(\psi_{1}, \psi_{2})^{\mathrm{T}}$ be  solutions of   \eqref{00} with $\lambda=i$  satisfying \eqref{case}.  Then $\varphi$ and $\psi$ are  linearly independent.
Further,  by Lemma \ref{constant}, \eqref{0041}, and  \eqref{case},  we get
\begin{equation}\label{414}
\begin{gathered}
 p(t-1)[\psi_1(t) \nabla\varphi_1(t)- \varphi_1(t)\nabla\psi_1(t)]\\
 +c(t-1) [\psi_1(t )\varphi_2(t-1)-\varphi_1(t )\psi_2(t-1) ]=-1,  \quad t\in \mathcal{I},
\end{gathered}
\end{equation}
which implies that\vspace{-2mm}
\begin{equation}\label{4122}
\begin{gathered}
\ \ \ \ \ \ \ \ \ \frac{p^{\frac{1}{2}}(t-1)}{  M^{\frac{1}{2}}(t)} \Big(|\psi_1(t) \nabla \varphi_1(t)|+|(\nabla \psi_1(t)) \varphi_1(t)| \Big)\\
\ \ \ \ \ \ \ \ \ \ \ \ \  + \frac{c^{\frac{1}{2}}(t-1)}{  M^{\frac{1}{2}}(t)}  \Big(|\psi_1(t)  \varphi_2(t-1)|+|\psi_2(t-1) \varphi_1(t)| \Big) \\
\geq \frac{1}{\left(p^{2}(t-1)+c^{2}(t-1)\right)^{\frac{1}{4}} M^{\frac{1}{2}}(t)}.
\end{gathered}\vspace{-2mm}
\end{equation}
If $\varphi, \psi \in l^{2}(\mathcal{I}),$ then $$\sum_{t=N}^{\infty} \displaystyle\frac{p(t-1)|\nabla \varphi_1(t)|^{2}}{M(t)}<\infty\ {\rm and}\ \sum_{t=N}^{\infty} \displaystyle\frac{p(t-1)|\nabla \psi_1(t)|^{2}}{M(t)}<\infty$$
by the above discussions. Then,  by the Cauchy's inequality and the first assumption  in 1), we  get from \eqref{4122} that\vspace{-2mm}
$$  \sum_{t=N}^{\infty} \frac{1}{\left(p^{2}(t-1 )+c^{2}(t-1 )\right)^{\frac{1}{4}} M^{\frac{1}{2}}(t)}<\infty, $$
which contradicts to assumption  4). Then $\mathcal{{L}}$ is in the  LPC  at $ t=\infty$. This
completes the proof.
\end{proof}

\begin{remark}
\quad
{\rm
\begin{itemize}
\item [{\rm (1)}] By Theorem 5.2,
${\mathcal{{L}}}^{(0)}$ is in the  LPC  at $ t=\infty$ under the following conditions: $p(t)>0$, 2) and 3) of Theorem 5.2, and
\begin{equation}\label{2101}
\sum_{t=N}^{\infty} \frac{1}{ {(p (t-1 )M (t))}^{  \frac{1}{2}}  }=\infty.  \vspace{-2mm}
\end{equation}
This limit point type of ${\mathcal{{L}}}^{(0)}$ is invariant under the perturbation  $\mathcal{{L}}^{(1)}$ in the case that conditions  1) and 4) of Theorem 5.2 hold.
\vspace{-2mm}
\item [{\rm (2)}] There were some limit point criteria in terms of coefficients for Sturm-Liouville differential   equations, e.g.,
    \cite{Everitt, Evans,Levinson}. Among them, there is   a well-known limit point criterion for Sturm-Liouville differential equations given by Levinson
 \cite[Theorem IV]{Levinson}.
Mingarelli \cite{Mingarelli} extended it to equation \eqref{1701} with $w(t)\equiv1 $  on ${\mathcal{I^{\prime}}} $,  i.e., equation \eqref{3}.
By \cite[Theorem 1]{Mingarelli}, if  there exist $k_{1}$, $k_{2}>0$, and $N\in \mathcal{I}$ such that     \vspace{-2mm}
\begin{align*}
&\ \ \ \  {\rm i})~q(t) \geq-k_{1} M(t), \ t>N, \hspace{12cm}\vspace{-2mm}\\
&\ \ \ \  {\rm ii}) ~ \frac{p^{\frac{1}{2}}(t-1)|\nabla M(t)|}{M^{\frac{1}{2}}(t) M(t-1)} \leq k_{2}, \ t>N,\vspace{-2mm}
\end{align*}
and \eqref{2101} holds,    then \eqref{1701}  is in the LPC  at $ t=\infty$. Clearly, Theorem 5.1 is a generalization of  \cite[Theorem 1]{Mingarelli} with  $w(t)\equiv1$ for Sturm-Liouville difference equation to matrix difference equation $(1.1_\lambda)$.

\end{itemize}
}
\end{remark}

\section{{{\sc\normalsize\bf Acknowledgements}}}\baselineskip 19pt
\ \ \ \
This research was supported by the NNSF of China (Grant 11971262), the NNSFs of Shandong Province (Grants ZR2020MA014   and ZR2019MA038).
%{Arlinskii1988, Arlinskii1995, Arlinskii1996, Arlinskii1998, Arlinskii1999, Arlinskii20001, Arlinskii2000, Arlinskii2002, Arlinskii2009}

\vspace{1cm}
 $\mbox{}\hfill\Box$
\end{document}